\documentclass{amsart}

\usepackage{amsmath, a4wide}
\usepackage{amssymb,amsthm,amsfonts}
\usepackage{amsrefs}
\usepackage[utf8]{inputenc}
\usepackage[T1]{fontenc}
\usepackage{xcolor}

\renewcommand{\mathcal}{\mathscr}

\def\R {\mathbb{R}}

\renewcommand{\epsilon}{\varepsilon}
\newcommand{\eps}{\varepsilon}
\newcommand{\e}{\varepsilon}

\renewcommand{\leq}{\leqslant}
\renewcommand{\le}{\leqslant}
\renewcommand{\geq}{\geqslant}
\renewcommand{\ge}{\geqslant}

\newtheorem{proposition}{Proposition}[section]
\newtheorem{theorem}[proposition]{Theorem}
\newtheorem{corollary}[proposition]{Corollary}
\newtheorem{lemma}[proposition]{Lemma}
\theoremstyle{definition}

\newtheorem{remark}[proposition]{Remark}
\numberwithin{equation}{section}

\title[Fractional symmetric self-shrinkers]{Symmetric self-shrinkers for the fractional mean curvature flow}

\author[A. Cesaroni, M. Novaga]{}


 \email{annalisa.cesaroni@unipd.it}
 \email{matteo.novaga@unipi.it}

\thanks{The authors were supported by the Italian INDAM--GNAMPA and by the University
of Pisa via grant PRA 2017 {\it Problemi di ottimizzazione e di evoluzione in ambito variazionale}.}

\begin{document}

\maketitle

\centerline{\scshape Annalisa Cesaroni}
\medskip
{\footnotesize
 \centerline{Department of Statistical Sciences}
   \centerline{University of Padova}
   \centerline{Via Cesare Battisti 141, 35121 Padova, Italy  }
} 

\medskip

\centerline{\scshape Matteo Novaga}
\medskip
{\footnotesize
 \centerline{ Department of Mathematics}
   \centerline{University of Pisa}
   \centerline{Largo Bruno Pontecorvo 5, 56127 Pisa, Italy  }
}

\medskip

\begin{abstract}
We show existence of  homothetically shrinking solutions of the fractional mean curvature flow, 
whose boundary consists in 
 a prescribed number of concentric spheres. 
  We prove that all these solutions, except from the ball,  are  dynamically unstable. 
 \end{abstract}


\section{Introduction}
Let us introduce the geometric evolution which we consider in this paper.
Given an initial set~$E\subset\R^n$, we define its evolution~$E_t$ according to
fractional mean curvature flow as follows:
the velocity at a point~$x\in \partial E_t$ is given by
\begin{equation}\label{kflow} 
\partial_t x\cdot \nu=-H_s(x,E_t):=-\lim_{\e\to 0}
\int_{\R^n\setminus B_\e(x)}\Big(
\chi_{\R^n\setminus E_t}(y)-\chi_{E_t}(y)\Big)\ \frac{1}{|x-y|^{n+s}}\,dy,\end{equation} 
where $s\in (0,1)$ is a fixed parameter and $\nu$ is the outer normal at $\partial E_t$ in $x$.  The fractional mean curvature of a set has been introduced in \cite{MR2675483} 
as the first variation of the fractional perimeter functional, and it has been proved in \cite{av}  that  for sufficiently smooth sets $E$ the rescaled fractional mean curvature 
 $(1-s)H_s(x, E)$ converges  as $s\to 1$ to the classical mean curvature of $E$ at $x$.
The evolution law \eqref{kflow} can be interpreted as the $L^2$-gradient flow of the fractional perimeter. 

Existence and uniqueness of viscosity solutions to a level set formulation of \eqref{kflow} has been provided in \cites{i, cmp}, and qualitative properties of  smooth solutions  have been studied in \cite{SAEZ}. However, we point out that the short-time existence of smooth solutions   has not yet been proved.
In \cite{cs}  the convergence to the  fractional mean curvature flow of a threshold dynamics scheme is proved;
this result was adapted to the anisotropic case, even in presence of a driving force in \cite{cnr},  where it is 
also shown that the flow  preserves convexity.  It has also been observed that the geometric law \eqref{kflow} presents some  different behavior with respect to the classical mean curvature flow: we refer for instance to the paper \cite{csv1} about the  formation of neck-pinch singularities,
and to the paper  \cite{cdnv} about  fattening and non-fattening phenomena.

In this paper we are interested in the homothetically shrinking  solutions for the flow \eqref{kflow}. 
A homothetic solution to \eqref{kflow} is a self-similar solution  to \eqref{kflow}: substituting  $E_t=\lambda(t) E$   in \eqref{kflow}, it is easy to see, using scale invariance of the fractional mean curvature,  that this is equivalent to  $\lambda' (t) x\cdot \nu=-\frac{1}{\lambda(t)^s}H_s(x,E)$ for all $x\in \partial E$. 
So homothetically shrinking solutions to \eqref{kflow} are given by the solutions to \eqref{kflow} with initial datum every set $E\subseteq \R^n$ of class $C^{1,1}$ which satisfies 
\begin{equation}\label{omo}
x\cdot \nu= c\, H_s(x, E)\qquad \text{for some constant $c>0$. }
\end{equation}  

 Homothetically shrinking solutions are particularly relevant in the analysis of the classical mean curvature flow, as they are {\it canonical} examples of singularities,
in the sense that any solution converges to a self-shrinker, if properly rescaled around a singular point. This result follows from an important {\it monotonicity formula} established by G. Huisken in \cite{MR1030675} for the mean curvature flow. The analog of such formula in the fractional setting is still an open problem. We recall moreover that, at the moment,
the existence theorem  for  local in time regular solutions of \eqref{kflow}, even if expected, has not been proved.  

It is well-known that the only embedded planar curve which is  homothetically shrinking under curvature flow is the circle \cite{MR845704}, whereas in higher dimensions there exist other  smooth embedded surfaces which are self-shrinkers for  the mean curvature flow, starting from  the  rotationally symmetric torus discovered by Angenent \cites{MR1167827}, and then going to more complex configurations as punctured  compact surfaces or  non-compact asymptotically conical surfaces, see \cites{MR1361726,kl}. However, it is easy to show that the ball is the only self-shrinker which is also radially symmetric.

In the fractional setting the classification of self-shrinkers is still at a very early stage. 
As far as we know, we provide here the first examples of fractional self-shrinkers which are different from  balls and cylinders.
More precisely, in Section \ref{sec:ex} we show the existence of 
homothetic solutions to the flow \eqref{kflow} which are radially symmetric, and have a prescribed number of boundary spheres (see Theorem \ref{omoproptris}). Moreover, in the case of a single annulus, we show uniqueness of the ratio ${R/r}$ for which the flow starting from the  annulus $B_R\setminus B_r$ self-similarly shrinks to a point.  
The existence of such  radially symmetric self-shrinkers, different from balls,  is a  new feature compared with the local case, and it is  due to the nonlocal nature
of the fractional mean curvature. 
 
\smallskip
A natural question arising about self-similar shrinkers is the issue of their dynamic stability. In the case of the classical mean curvature flow, the study of the dynamic stability of self-shrinkers was initiated in \cite{cm},
and later developed by other authors. From the convergence results in \cites{MR840401,MR772132}
it follows that the balls is dynamically stable under mean curvature flow (see also \cites{MR1637988,np,pp} for a discussion of the stability of the Wulff-Shape as homothetic solution of the anisotropic and crystalline curvature flow). Moreover,  in \cite{cm}
it is shown that balls and cylinders are the only stable self-shrinkers.

In the fractional case none of such results is currently available, in particular it is not known whether the ball is dynamically stable, and if convex sets shrink to a round point at the singular time. We discuss  in this paper the stability issue for the class of solutions that we construct in Theorem \ref{omoproptris}. 
In particular, in Section \ref{sec:stab} we show that the radial self-shrinkers different from the ball are all dynamically unstable (see Theorem \ref{teostab}). 

\smallskip

\noindent {\bf Acknowledgements.} The authors are members of INDAM-GNAMPA.
The second author was partially supported by the University of Pisa Project PRA 2017 "Problemi di ottimizzazione e di evoluzione in ambito variazionale".

\section{Existence of symmetric self-shrinkers} \label{sec:ex} 
We start with a technical result which will be useful in the sequel.
We denote by $B_r$ the ball of center $0$ and radius $r>0$, and we let $B_r(x)=x+B_r$.  
Moreover, we recall that, by  the scale invariance of fractional mean curvature, for all $x_r\in \partial B_r$
there holds  (see \cite[Lemma 2]{SAEZ})
\[
H_s(x_r, B_r)=\frac{k(n)}{r^s}\quad\text{where }  k(n):=H_s(x_1, B_1).
\]  

\begin{lemma}\label{lemtech}
Let $x\in \R^n\setminus\{0\}$ and $\delta\ne 0$. Then, as $\delta\to 0$, the following estimate holds:
\[
\int_{\partial B_{|x|+\delta}} \frac{1}{|y-x|^{n+s}}\,dy =  \frac{1}{|\delta|^{1+s}} (c+ o(1)) \,,
\]
for a constant $c>0$ depending only on $n$ and $s$. 
\end{lemma}

\begin{proof}
Up to a rotation of the reference system, we can assume that $x=-|x|e_1$. 
By the change of coordinates $y' = (y-x) $, we get,  
\[
\int_{\partial B_{|x|+\delta}} \frac{1}{|y-x|^{n+s}}\,dy = \int_{\partial B_{|x|+\delta}(|x|e_1)}  \frac{1}{|y'|^{n+s}}\,dy'\]
Note that
\begin{equation}\label{unolemma} 
\int_{\partial B_{|x|+\delta}(|x|e_1)\cap \{y'\cdot e_1>|x|\}}  \frac{1}{|y'|^{n+s}}\,dy'\leq \frac{n\omega_{n}\,(|x|+\delta)^{n-1}}{2\,|x|^{n+s}}\leq C.\end{equation} 
Moreover we write $\partial B_{|x|+\delta}(|x|e_1)\cap \{y'\cdot e_1<|x|\}=\{(f(z),z)\, |\, z\in \R^{n-1}, z\in B'_{|x|+\delta}\}$,  
where  $B'_r\subseteq \R^{n-1}$ denotes the ball of center $0$ and radius $r$ in $\R^{n-1}$  and $f(z)=|x|-\sqrt{(|x|+\delta)^2-|z|^2}$. 
Therefore, denoting $R_\delta:=\frac{|x|+\delta}{|\delta|}$, we get
\begin{eqnarray*} \nonumber 
&&  \int_{\partial B_{|x|+\delta}(-x)\cap \{y'\cdot e_1<|x|\}}  \frac{1}{|y'|^{n+s}}\,dy' =
 \int_{B'_{|x|+\delta}}\frac{|x|+\delta}{\sqrt{(|x|+\delta)^2-|z|^2}} \frac{1}{\left(f(z)^2+|z|^2 \right)^{\frac{n+s}{2}}}dz 
\\ 
&=&\frac{ (n-1)\omega_{n-1} }{|\delta|^{s+1}} \int_0^{R_\delta} \frac{R_\delta}{\sqrt{R_\delta^2-\rho^2}} 
\rho^{n-2}\left[\left( \frac{\delta}{|\delta|}-\frac{\rho^2}{R_\delta+\sqrt{R_\delta^2-\rho^2}}\right)^2+\rho^2 \right]^{-\frac{n+s}{2}}d\rho.\label{trelemma}
\end{eqnarray*}
Let
\[g_\delta(\rho):= \frac{R_\delta}{\sqrt{R_\delta^2-\rho^2}} 
\rho^{n-2}\left[\left( \frac{\delta}{|\delta|}-\frac{\rho^2}{R_\delta+\sqrt{R_\delta^2-\rho^2}}\right)^2+\rho^2 \right]^{-\frac{n+s}{2}}.\]  
Now we observe that there exists $C=C(n,s)>0$, such that 
\begin{equation}\label{duelemma} \int_{R_\delta/2}^{R_\delta} g_\delta(\rho)d\rho\leq \frac{C}{R_\delta^{s+1}}= C\frac{|\delta|^{s+1}}{(|x|+\delta)^{s+1}}.\end{equation} 
Moreover, taking $|\delta|$ sufficiently small such that $R_\delta>4$, we get that there exists a dimensional constant $C=C(n,s)>0$ such that 
\begin{equation}\label{quattrolemma} g_\delta(\rho)\chi_{(0,R_\delta/2)}(\rho)\leq C \chi_{(0,1)}(\rho)+C \frac{1}{\rho^{s+2}}\chi_{(1, +\infty)}(\rho)\in L^1(0, +\infty), \end{equation} 
where $\chi_{(a,b)}$ is the characteristic function of the interval $(a,b)$.  Using \eqref{quattrolemma}  and observing that \[g_\delta(\rho)  \to   \frac{\rho^{n-2}}{(1+\rho^2)^\frac{n+s}{2}}\qquad \text{as $\delta\to 0$}, \] 
we conclude by \eqref{duelemma} and by Lebesgue dominated convergence theorem that 
\[ \int_{\partial B_{|x|+\delta}(-x)\cap \{y_1'<|x|e_1\}}  \frac{1}{|y'|^{n+s}}\,dy'= \frac{ (n-1)\omega_{n-1} }{|\delta|^{s+1}}\left[ \int_0^{\frac{R_\delta}{2}}g_\delta(\rho)d\rho+  \int_{\frac{R_\delta}{2}}^{R_\delta}g_\delta(\rho)d\rho\right]=
\frac{c+o(1)}{|\delta|^{s+1}} \]  where \[c :=  (n-1)\,\omega_{n-1}
\int_0^{+\infty}\frac{\rho^{n-2}}{(1+\rho^2)^\frac{n+s}{2}}\,d\rho\,.
\] The conclusion follows by this estimate and \eqref{unolemma}. 
\end{proof}
First of all we look to the  simplest example of rotationally symmetric set different from a ball. We show that there exists a unique value of the ratio $\frac{R}{r}$ which depends on the dimension $n$ and on the fractional power $s\in (0,1)$ such that the annulus $B_R\setminus B_r$ is a self-shrinker. 
 
 \begin{proposition}\label{omoprop}  
Let $n\geq 1$. Then, for all $R>0$ fixed there exists a unique  $r=r(n,s)\in (0,R) $ depending only on $R$,  $s\in (0,1)$ and $n$, such that the flow \eqref{kflow} with initial datum the  annulus 
\[
A:= B_R\setminus B_r
\] 
is a homothetically shrinking  solution of the flow.
\end{proposition}

\begin{proof} Up to rescaling the set we fix $R=1$. 
We observe that $A$ is a solution to \eqref{omo} if and only if for some $c>0$, 
\[
 1= c H_s(x_1,  A)\text{ for all $x_1$ with $|x_1|=1$}\quad \text{ and } \quad r=-cH_s(x_r, A) \text{ for all $x_r$ with $|x_r|=r$}\]
and so if and only if 
\begin{equation}\label{omo1} H_s(x_r, A)=-r H_s(x_1, A).
\end{equation}
By rotational invariance, we get that $H_s(x_r, A), H_s(x_1, A)$ do not depend on the points $x_r, x_1$, but only on $0<r<1$. Moreover they are both continuous functions with respect to $r$,  due to  the continuity  of the fractional mean curvature with respect to $C^2$-convergence of sets (see \cite[Section 5.2]{cmp}).
We consider the following function  defined for $r\in (0, 1)$ 
\begin{equation}\label{funzionef} f_s(r)= H_s(x_r, A)+r  H_s(x_1, A).\end{equation} Note that the function $f_s$ is continuous on $(0,1)$. To prove the statement it is sufficient to show that there exists a unique $r=r(n,s)$ such that $f_s(r(n,s))=0$.  

Let $r, r'$ such that $ 0<r<r'<1$.  By the inclusions $A_{1,r'}:= B_1\setminus B_{r'}\subseteq A\subseteq B_1$, we get, by the monotonicity of the fractional mean curvature (see \cite[Section 5.2]{cmp}),  
that  \begin{equation}\label{monofuori} H_s(x_1, A_{1,r'})>H_s(x_1, A)\geq H_s(x_1, B_1)= k(n) >0.\end{equation}  This implies that  
\begin{equation}\label{uno} 
r\in (0, 1)\mapsto H_s(x_1, A)\text{ is monotone increasing  and positive}.
\end{equation}

Moreover,  we observe, recalling the definitions, that 
\[H_s(x_r, A)=-H_s(x_r, B_r)+2\int_{\R^n\setminus B_1} \frac{1}{|x_r-y|^{n+s}}dy.\] 
Note that if $r'>r$ then $H_s(x_{r'}, B_{r'})= \frac{k(n)}{(r')^{s}}< \frac{k(n)}{r^s}= H_s(x_r, B_r)$, whereas for $1>r'>r$,  $|x_r-y|\geq |\frac{r'}{r} x_{r}-y|$ for all $y\in \R^n\setminus B_1$, and $x_ r$ with $|x_r|=r$. Therefore  by symmetry  of the kernel we have that \[\int_{\R^n\setminus B_1} \frac{1}{|x_r-y|^{n+s}}dy< \int_{\R^n\setminus B_1} \frac{1}{|x_{r'}-y|^{n+s}}dy\] for all $x_r, x_{r'}$  with $|x_r|=r, |x_{r'}|=r'$. 
Using these facts we conclude that 
  \begin{equation}\label{due} r\in (0, 1)\mapsto H_s(x_r, A)\text{ is monotone increasing}.
\end{equation}

Due to \eqref{uno}, \eqref{due}, we notice that the function $f_s(r)$ defined in \eqref{funzionef} is monotone increasing. 
Now we claim that  $\lim_{r\to 0} f_s(r)=-\infty$ and that $\lim_{r\to 1}f_s(r)=+\infty$.  If the claim is true, then the proof is concluded. 
 
 First of all we observe  that 
 \[ H_s(x_r, A)= \int_{\R^n\setminus B_{1}} \frac{1}{|x_r-y|^{n+s}}dy+\lim_{\eps\to 0}\left(\int_{B_{r-\eps}} \frac{1}{|x_r-y|^{n+s}}dy-\int_{B_1\setminus B_{r+\eps} }\frac{1}{|x_r-y|^{n+s}}dy\right).  \]This  implies that $\lim_{r\to 0} H_s(x_r, A)=-\infty$, and so also $\lim_{r\to 0} f_s(r)=-\infty$.

Moreover, recalling Lemma \ref{lemtech} we get that 
\begin{multline*}  H_s(x_r, A) =H_s(x_r, B_r)+\lim_{\eps\to 0}\left(2\int_{B_{r-\eps}} \frac{1}{|x_r-y|^{n+s}}dy-2\int_{B_1\setminus B_{r+\eps} }\frac{1}{|x_r-y|^{n+s}}dy\right) \\
 =\frac{k(n)}{r^s}+\lim_{\eps\to 0}\left(2\int_\eps^{r }\int_{\partial B_{r-\delta}} \frac{1}{|x_r-y|^{n+s}}dyd\delta-2\int_{\eps}^{1-r}\int_{\partial B_{r+\delta}}
 \frac{1}{|x_r-y|^{n+s}}dyd\delta \right)\\
 =\frac{k(n)}{r^s}+  2\int_{1-r}^{r}\left[\int_{\partial B_{r-\delta}} \frac{1}{|x_r-y|^{n+s}}dy-  \int_{\partial B_{r+\delta}}
 \frac{1}{|x_r-y|^{n+s}}dy\right]d\delta  \\
=\frac{k(n)}{r^s}+  2\big(c + o(1)\big) 
\left( \frac{1}{(1-r)^s}-\frac{1}{r^s}\right). \end{multline*} 
 So, $\lim_{r\to 1} H_s(x_r, A)=+\infty$, which permits to conclude that $\lim_{r\to 1} f_s(r)=+\infty$. 
\end{proof}

We now look for more general symmetric self-shrinkers, given by the union of a finite number of annuli.

 \begin{theorem}\label{omoproptris}  
For all $N\ge 1$ and all $R>0$ there exists  an increasing sequence 
$0< r_1< \ldots<r_{2N-1}<r_{2N}=R$, depending only on $n$, $s$ and $N$, such that 
such that the flow \eqref{kflow} with initial datum
\[
E:= \bigcup_{k=1}^N \left(B_{r_{2k}}\setminus B_{r_{2k-1}}\right)
\] 
is a homothetically shrinking  solution of \eqref{kflow}. 

Similarly, for all $N\ge 1$ and $R>0$ there exists an increasing sequence 
$0< \tilde r_0<\tilde r_1<\ldots<\tilde r_{2N-1}< r_{2N}=R$, depending only on $n$, $s$ and $N$, such that 
such that the flow \eqref{kflow} with initial datum
\[
\widetilde E:= B_{\tilde r_0}\cup\, \bigcup_{k=1}^N \left(B_{\tilde r_{2k}}\setminus B_{\tilde r_{2k-1}}\right)
\] 
is a homothetically shrinking  solution of \eqref{kflow}.
\end{theorem}

 \begin{proof} The argument is similar to that in the proof of Proposition \ref{omoprop}. 
As before, up to rescaling the sets $E$, $\widetilde E$, we can assume $r_{2N}=1$. Then, 
we want to find  radii $r_i$ 
in such a way that, letting $x_{r_i}\in  \partial B_{r_i}$, there hold
\begin{equation}\label{eqri}
f_i(r_1,\ldots,r_{2N-1}):= r_i H_s(x_1,E) + (-1)^{i-1} H_s(x_{r_i},E)=0
\qquad \forall\ i\in \{1,\ldots,2N-1\}  
\end{equation} 
and \begin{equation}\label{eqri2}
f_i(r_0,\ldots,r_{2N-1}):= r_i H_s(x_1,\widetilde E) + (-1)^{i-1} H_s(x_{r_i},\widetilde E)=0
\qquad \forall\ i\in \{0,\ldots,2N-1\}.
\end{equation} 
Notice that the functions $f_i$ are all continuous in their domain of definition.

We divide the proof into 4 steps. In the first step we deal with the case $N=1$ , and in  step 2, 3 and 4 we consider the case $N>1$.
For $N>1$ we provide the proof just of \eqref{eqri} for the existence of the set $E$, since the analogous assertion \eqref{eqri2} for $\widetilde E$ follows similarly. 

\vspace{0,3cm}

\noindent{\it Step 1.}  
The case $N=1$   for $E$ has been proved in  Propositions \ref{omoprop}.  So, we consider the set $\tilde E$. 

First of all we fix $r_1\in (0,1)$ and we prove that there exists $r_0=r_0(r_1)\in (0, r_1)$ such that $f_0(r_0(r_1), r_1)=0$ for all $r_1\in (0,1)$. 
Due to the monotonicity properties of the  fractional mean curvature, fixed  $r_1\in (0,1)$ we get
\[\lim_{r_0\to 0}  H_s(x_{1},\tilde E)=H_s(x_1, A_{1,r_1})>0\qquad  \lim_{r_0\to r_1} H_s(x_{1},\tilde E)=H_s(x_{1}, B_1)=k(n)>0.\]
Moreover, by definition we get that, 
\[ H_s(x_{r_0},\tilde E)=-2\int_{B_{1}\setminus B_{r_1}} \frac{1}{|x_{r_0}-y|^{n+s}}\,dy +
H_s(x_{r_0}, B_{r_0})= -2\int_{B_{1}\setminus B_{r_1}} \frac{1}{|x_{r_0}-y|^{n+s}}\,dy +\frac{k(n)}{r_0^s}, \]
from which we conclude that 
\[ \lim_{r_0\to 0}H_s(x_{r_0},\tilde E)=+\infty \qquad   \lim_{r_0\to r_1}H_s(x_{r_0},\tilde E)=-\infty.\]
Therefore, we obtain that
\[\lim_{r\to 0} f_0(r, r_1)=-\infty\qquad \lim_{r\to r_1} f_0(r, r_1)=+\infty.\]

By continuity of the function $f_0$, we deduce that for all $r_1 \in (0,1)$ there exists at least one $r=r(r_1)\in (0, r_1)$ such that 
\begin{equation}\label{f0r}
f_0(r(r_1), r_1)=0\,.
\end{equation}  
We choose as $r_0(r_1)$ to be the smallest among all possible $r(r_1)\in (0, r_1)$ which solve \eqref{f0r}.
Observe that due to this choice the function $r\to f_1(r_0(r), r)$ is continuous.    To conclude it is sufficient to prove that  that there exists $r_1\in (0,1)$ such that $ f_1(r_0(r_1), r_1)=0$. 
Indeed, this would imply that $(B_1\setminus B_{r_1})\cup B_{r_0(r_1)}$ is a solution to \eqref{omo}. 

Observe that $\lim_{r\to 0}r_0(r)=0$, and therefore we get
\begin{equation}\label{eqg1}
\lim_{r\to 0}  f_1(r_0(r), r)= -\infty.
\end{equation}

We now claim that 
\begin{equation}\label{eqg2}
\lim_{r\to 1}  f_1(r_0(r), r)= +\infty.
\end{equation}
Recalling Lemma \ref{lemtech}, we observe that as $r\to 1$,  
\begin{eqnarray}\nonumber
H_s(x_{1}, \tilde E) &=& 
2\int_{B_{r}\setminus B_{r_0(r)}} \frac{1}{|x_1-y|^{n+s}}\,dy +
H_s(x_1, B_1) 
= 2\int_{r_0(r)}^{r}\int_{\partial B_t} \frac{1}{|x_1-y|^{n+s}}\,dy\,dt+k(n)
\\ \label{cir1}
&=& 2\big(c + o(1)\big) \left( \frac{1}{(1-r)^s}- \frac{1}{(1-r_0(r))^s}\right) +k(n)
\end{eqnarray}
where the constant $c=c(n,s)>0$ is given by Lemma \ref{lemtech}.
Similarly, we have that 
\begin{eqnarray}\nonumber
H_s(x_{r},\tilde E) & =&\lim_{\eps\to 0}\left( 2\int_{B_{r-\eps}\setminus B_{r_0(r)}} \frac{1}{|x_{r}-y|^{n+s}}\,dy 
 -2\int_{B_{1}\setminus B_{r+\eps}} \frac{1}{|x_{r}-y|^{n+s}}\,dy\right) +H_s(x_{r}, B_{r})
\\ \nonumber &=& \lim_{\eps\to 0}\left(2 \int_{r_0(r)}^{r-\eps}\int_{\partial B_t}\frac{1}{|x_{r}-y|^{n+s}}\,dy-2\int_{r+\eps}^1 \int_{\partial B_t}\frac{1}{|x_{r}-y|^{n+s}}\,dy\right)+\frac{k(n)}{r^s}
\\\label{cir2} 
 &=& 2 \big(c + o(1)\big) \left(- \frac{1}{(r-r_0(r))^s}
+\frac{1}{(1-r)^s}\right)+\frac{k(n)}{r^s}
\end{eqnarray}
and 
\begin{eqnarray}\nonumber
H_s(x_{r_0(r)},\tilde E) &= &  - 2\int_{r_n}^ {1}\int_{\partial B_t}\frac{1}{|x_{r_0(r)}-y|^{n+s}}\,dy+ \frac{k(n)}{(r_0(r))^s}\\ \label{cir3} &=& - 2\big(c + o(1)\big) \left(  \frac{1}{(r-r_0(r))^s}- \frac{1}{(1-r_0(r))^s} \right)+ \frac{k(n)}{(r_0(r))^s}.\end{eqnarray} 
Therefore as $r\to 1$   by \eqref{cir1} and \eqref{cir2} 
\begin{equation}\label{f1nuova}  f_1(r_0(r), r)=
 2\big(c + o(1)\big) \left(\frac{1+r}{(1-r)^s} -\frac{1}{(r-r_0(r))^s}\ - \frac{r }{(1-r_0(r))^s}\right) + O(1).\end{equation}
We claim that
\begin{equation}\label{claim1} 
\lim_{r\to 1}\frac{r-r_0(r) }{1-r_0(r)}=1.\end{equation}
Note that the claim is equivalent to 
\[\lim_{r\to 1} \frac{1-r}{1-r_0(r)}=0 =\lim_{r\to 1} \frac{1-r}{r-r_0(r)}\]
and this implies immediately, recalling \eqref{f1nuova}, that $\lim_{r\to1} f_1(r_0(r),r)=+\infty$.

To prove \eqref{claim1} we recall  that $f_0(r_0(r),r)=0$  and using \eqref{cir1} and \eqref{cir3} we get 
\[2\big(c + o(1)\big) \left( \frac{r_0(r)}{(1-r)^s}- \frac{r_0(r)+1}{(1-r_0(r))^s}+ \frac{1}{(r-r_0(r))^s} \right)+r_0(r)k(n)+ \frac{k(n)}{(r_0(r))^s}=0\]
from which we deduce that
\begin{equation}\label{ester}
\frac{r_0(r)}{(1-r)^s}+\frac{1}{(r-r_0(r))^s} =  \frac{1+r_0(r)}{(1-r_0(r))^s} + O(1).
\end{equation}
Recalling that 
\[
\frac{1}{(1-r)^s}\ge  \frac{1}{(1-r_0(r))^s}\,\quad\text{and}\quad  \frac{1}{(r-r_0(r))^s}\ge  \frac{1}{(1-r_0(r))^s}\,
\]
from \eqref{ester} we get that
\[
\frac{1}{(1-r_0(r))^s}\le  \frac{1}{(r-r_0(r))^s} \le \frac{1}{(1-r_0(r))^s} + O(1)\,,
\]
which gives the claim \eqref{claim1}. 

\smallskip
 
By continuity of $f_1$, from \eqref{eqg1} and \eqref{eqg2} it follows that there exists $r_1\in (0,1)$ such that
 $ f_1(r_0(r_1), r_1)=0$, which gives the thesis.
 
\vspace{0,3cm}

\noindent{\it Step 2.}  We pass now to consider the case $N>1$. 
We provide a proof of the existence of a sequence of radii $r_i$ which solves \eqref{eqri}. 
We shall determine $r_i$ by induction on $i$.

 For $i=1$ we observe that, given a choice of $0<r_2<\ldots< r_{2N-1}<1$, we have
\[
\lim_{r_1\to 0}H_s(x_{r_1},E) = -\infty\quad\text{and}\quad \lim_{r_1\to r_2}H_s(x_{r_1},E) = +\infty.
\]
By continuity of the function $f_1$ 
it follows that there exists $\bar r_1=\bar r_1(r_2, \ldots r_{2N-1}) \in (0,r_2)$ such that 
$f_1(\bar r_1,\ldots,r_{2N-1})=0$. As before,
in case of multiple solutions we choose the smallest one.
Notice that $\bar r_1$ is continuous as a function of $r_2,\ldots, r_{2N-1}$.
Notice also that, if we fix $r_3,\ldots,r_{2N-1}$ and let $r_2\to r_3$,
letting $F:= B_{\bar r_1} \cup A_{r_3,r_2}$ and
proceeding as in  Step 1, we get
\begin{eqnarray*}
H_s(x_{\bar r_1},E) &=& - H_s(x_{\bar r_1},F) + O(1) =2 \big( c+o(1)\big) \left( \frac{1}{|r_2-\bar r_1|^{n+s}} 
- \frac{1}{|r_3-\bar r_1|^{n+s}}\right) + O(1)
\end{eqnarray*}
Since $f_1(\bar r_1,\ldots,r_{2N-1})=0$, we also have $H_s(x_{\bar r_1},E)=-\bar r_1 H_s(x_{\bar r_1},E) = O(1)$, whence
\begin{equation}\label{lolli}
\lim_{r_2\to r_3} \frac{|r_2-\bar r_1|}{|r_3-\bar r_1|} = 1.
\end{equation}

\vspace{0,3cm}

\noindent{\it Step 3.}  Let now $2\le i<2N-1$. By induction assumption, for all $j<i$ there exist continuous
functions $\bar r_j(r_i,\ldots,r_{2N-1}))$ such that  $f_j(\bar r_1,\ldots, \bar r_{i-1}, r_i, \ldots, r_{2N})=0$.
In view of \eqref{lolli}, we shall also assume that
\begin{equation*}
\lim_{r_i\to r_{i+1}} \frac{|r_i-\bar r_{i-1}|}{|r_{i+1}-\bar r_{i-1}|} = 1,
\end{equation*}
which is equivalent to
\begin{equation}\label{ruggi}
\lim_{r_i\to r_{i+1}} \frac{|r_{i+1}-r_{i}|}{|r_{i}-\bar r_{i-1}|} = 0.
\end{equation}
Given a choice of $r_j$ for $j>i$, we want to find $\bar r_i$ such that  
\begin{equation}\label{bardo}
f_i(\bar r_1,\ldots, \bar r_{i}, r_{i+1}, \ldots, r_{2N})=0
\end{equation}
and 
\begin{equation}\label{bardi}
\lim_{r_{i+1}\to r_{i+2}} \frac{|r_{i+1}-\bar r_{i}|}{|r_{i+2}-\bar r_{i}|} = 1.
\end{equation}
We first notice that 
\[
\lim_{r_i\to 0} f_i(\bar r_1,\ldots, \bar r_{i-1}, r_i, \ldots, r_{2N-1}) 
= \lim_{r_i\to 0}  (-1)^{i-1} H_s(x_i,E)
= -\infty.
\]
We now consider the limit $r_i\to r_{i+1}$. Reasoning
as in Step 1, we get
\begin{eqnarray*}
(-1)^{i-1} H_s(x_{r_i},E) &=& 2\big( c+o(1)\big) \left( \frac{1}{|r_{i+1}-r_i|^s} 
+\sum_{j=k}^{i-1} \frac{ (-1)^{i-k}}{|r_i-\bar r_k|^s} \right) + O(1),
\end{eqnarray*}
and therefore, recalling \eqref{ruggi},
\begin{eqnarray*}
\lim_{r_i\to r_{i+1}} f_i(\bar r_1,\ldots, \bar r_{i-1}, r_i, \ldots, r_{2N-1}) 
&=& \lim_{r_i\to r_{i+1}}  (-1)^{i-1} H_s(x_{r_i},E) +O(1) 
\\
&=& \lim_{r_i\to r_{i+1}}   \left(\frac{1}{|r_{i+1}-r_i|^s} 
+\sum_{j=k}^{i-1} \frac{ (-1)^{i-k}}{|r_i-\bar r_k|^s}\right) = +\infty.
\end{eqnarray*}
By continuity of $f_i$ 
it follows that there exists $\bar r_i$ such that 
$f_i(\bar r_1,\ldots, \bar r_{i}, r_{i+1}, \ldots, r_{2N})=0$. As before,
in case of multiple solutions we choose the smallest one.

We now show \eqref{bardi}. If we fix $r_{i+2},\ldots,r_{2N-1}$ and let $r_{i+1}\to r_{i+2}$, 
from \eqref{bardo} we get $H_s(x_{\bar r_i},E) = O(1)$, which implies
\begin{eqnarray*}
-\frac{1}{|r_{i+2}-\bar r_{i}|^s} + \frac{1}{|r_{i+1}-\bar r_i|^s} 
+\sum_{j=k}^{i-1} \frac{ (-1)^{i-k}}{|\bar r_i-\bar r_k|^s} = O(1).
\end{eqnarray*}
Multiplying by $|r_{i+1}-\bar r_i|^s$ and recalling \eqref{ruggi}
we then get
\begin{eqnarray*}
\lim_{r_{i+1}\to r_{i+2}} \frac{|r_{i+1}-\bar r_i|^s}{|r_{i+2}-\bar r_{i}|^s}
- \sum_{j=k}^{i-1} \frac{ (-1)^{i-k}|r_{i+1}-\bar r_i|^s}{|\bar r_i-\bar r_k|^s} 
= \lim_{r_{i+1}\to r_{i+2}} \frac{|r_{i+1}-\bar r_i|^s}{|r_{i+2}-\bar r_{i}|^s}
= 1,
\end{eqnarray*}
which gives \eqref{bardi}.

\vspace{0,3cm}

\noindent{\it Step 4.} Finally, for $i=2N-1$ we still have
\[
\lim_{r_{2N-1}\to 0} f_{2N-1}(\bar r_1,\ldots, \bar r_{2N-2}, r_{2N-1}) 
= -\infty.
\]
We now consider the limit $r_{2N-1}\to 1$. Recalling \eqref{bardi} with $i=2N-2$, as in Step 1 we get
\begin{eqnarray*}
H_s(x_{r_{2N-1}},E) &=&2 \big( c+o(1)\big) \left( \frac{1}{(1-r_{2N-1})^s} 
- \frac{1}{(r_{2N-1}-\bar r_{2N-2})^s}\right) + O(1)
\\ 
H_s(x_{1},E) &=& 
2\big( c+o(1)\big) \left( \frac{1}{(1-r_{2N-1})^s} 
- \frac{1}{(1-\bar r_{2N-2})^s}\right) + O(1)
\\
&=& 
2\big( c+o(1)\big) \, \frac{1}{(1-r_{2N-1})^s} + O(1),
\\
&=& 
2\big( c+o(1)\big) \, \frac{1}{(1-r_{2N-1})^s} + O(1).
\end{eqnarray*}
Therefore, we have
\begin{eqnarray*}
0 &=& \lim_{r_{2N-1}\to 1} f_{2N-1}(\bar r_1,\ldots, \bar r_{2N-2}, r_{2N-1}) 
\\
&=& \lim_{r_{2N-1}\to 1} \left( r_{2N-1} H_s(x_{1},E) + H_s(x_{r_{2N-1}},E)\right)
\\
&=& \lim_{r_{2N-1}\to 1} 2\big( c+o(1)\big) , \frac{1+r_{2N-1}}{|1-r_{2N-1}|^s} = +\infty.
\end{eqnarray*}
As before, it follows that there exists $\bar r_{2N-1}$ such that 
$f_{2N-1}(\bar r_1,\ldots, \bar r_{2N-1}) =0$. 
\end{proof}

\begin{remark}
\upshape An interesting question which is left open by the previous result is the issue of uniqueness for self-shrinkers with  a prescribed number of boundary spheres. In the simplest case, that is the annulus, in Proposition \ref{omoprop} we prove uniqueness of the ratio $\frac{R}{r}$ for which the annulus $B_R\setminus B_r$ is a self-similar shrinker.
\end{remark}
 
{}From Theorem \ref{omoproptris} we readily obtain the existence of cylindrical  self-shrinkers. 

 \begin{corollary}\label{omocil}  Let $k<n$. 
For all $N\ge 1$  and $R>0$ there exists  an increasing sequence 
$0< r_1<\ldots<r_{2N-1}<r_{2N}=R$, depending only on $k$, $s$ and $N$, such that 
such that the flow \eqref{kflow} with initial datum
\[
C:= \R^{n-k}\times  \bigcup_{j=1}^N \left(B^k_{r_{2j}}\setminus B^k_{r_{2j-1}}\right)
\] 
is a homothetically shrinking  solution of \eqref{kflow}, where $B^k_r$ denotes the ball of radius $r$ in $\R^k$. 

Similarly, for all $N\ge 1$ and $R>0$ there exists an increasing sequence 
$0< \tilde r_0<\tilde r_1<\ldots<\tilde r_{2N-1}<\tilde r_{2N}=R$, depending only on $k$, $s$ and $N$, such that 
such that the flow \eqref{kflow} with initial datum
\[
\widetilde C:=\R^{n-k}\times B^k_{\tilde r_0}\cup\, \bigcup_{j=1}^N \left(B^k_{\tilde r_{2j}}\setminus B^k_{\tilde r_{2j-1}}\right)
\] 
is a homothetically shrinking  solution of \eqref{kflow}.
\end{corollary}

\begin{remark}\upshape\label{rem} 
We observe that  the radii $r(n,s)$ in Proposition \ref{omoprop}, $r_i(n,s), \tilde r_i(n,s)$ in Theorem \ref{omoproptris} and Corollary \ref{omocil} all  satisfy $\lim_{s\to 1} r(n,s)=\lim_{s\to 1}   r_i(n,s)=\lim_{s\to 1} \tilde r_i(n,s)=R$.  

We give a brief justification of this fact just for the simplest case, that is the case of $r(n,s)$ in Proposition \ref{omoprop}, the others being completely analogous. 
 We recall that if $E\subset\R^n$ is a compact set with $C^2$ boundary then 
$(1-s)H_s(x,E)$ converges uniformly as $s\to 1$ to  the classical mean curvature $H(x,\partial E)$ (see \cite{av}). Under  the same notation as in the proof of Proposition \ref{omoprop}, we note that for $s=1$ the function $f_1(r)$ defined in \eqref{funzionef} is given by $r-\frac{1}{r}$ (this is also true for the functions $f_i$ defined in the proof of Theorem \ref{omoproptris}, that is $f_i= r_i-\frac{1}{r_i}$). So, by uniform convergence of the curvatures, we get that if   $(r_k)_k$ is a sequence with $r_k\in (0,1)$ and $r_k\to 1$,  there exists $(s_k)_k$ with $0<s_k<1$ such that $f_s(t)<0$ for $t\in (0, r_k]$ and $s\geq s_k$. This  implies that $r(n, s)>r_k$ for all $s\geq s_k$ and that $s_k\to 1$, since $\lim_{r\to 1}f_s(r)=+\infty$ for all $s<1$.    \end{remark} 

\section{Stability}\label{sec:stab}
We now discuss the dynamic stability of the symmetric self-shrinkers constructed in the previous section.
By definition self-shrinkers are stationary solutions to the flow 
\begin{equation}\label{sflow} 
\partial_t x\cdot \nu=-H_s(x,E)+x\cdot \nu. 
\end{equation}
If the initial datum is rotationally symmetric as in Theorem \ref{omoproptris} then \eqref{sflow} becomes a system of ODE's
in the radii $r_i$, and Theorem   \ref{omoproptris}  guarantees the existence of a stationary point for every number of radii.
We are interested in the stability of such critical points, with respect to perturbations which are orthogonal to the 
vector $(r_1,\ldots r_{2N})$ (or resp. $(r_0,\ldots r_{2N})$) given by the radii. Indeed this vector corresponds to a rescaling of the initial datum, and therefore gives a direction of instability for the system which is not geometrically significant.

In the symmetric situation, we can rewrite \eqref{sflow} as the system of ODE's
\begin{equation}\label{syst} 
\dot r_i= (-1)^{i-1}H_s(x_i,E)+r_i  \qquad i\le 2N.
\end{equation}

\begin{theorem}\label{teostab} 
Fix $N\ge 1$, and let $E$ (resp. $\tilde E$) be the symmetric shrinker given by Theorem \eqref{omoproptris},
corresponding to the stationary point $(\bar r_1,\ldots \bar r_{2N})$ 
(resp. $(\bar r_0,\ldots \bar r_{2N})$)
for the system \eqref{syst}.
Then, the Morse index of such point is at least $2$, in particular the corresponding homothetic solution is dynamically unstable.
\end{theorem}

\begin{proof} 
We shall prove the assertion for the shrinker $E$, since the proof for $\tilde E$ is analogous.

For the reader convenience, we first present in detail the case $N=1$, corresponding to an annulus $A=B_{\bar r_2}\setminus B_{\bar r_1}$. The system \eqref{syst} then becomes
\begin{equation}\label{sistemaanello}
\begin{cases}
\dot {r_1}= H_s(x_{r_1}, A)+r_1= \frac{k(n)}{r_1^s}+2\lim_{\eps\to 0}\left(\int_{B_{r_1-\eps}} \frac{1}{|x_{r_1}-y|^{n+s}}dy-
 \int_{B_{r_2}\setminus B_{r_1+\eps}} \frac{1}{|x_{r_1}-y|^{n+s}}dy\right)+r_1\\  
\dot r_2= -H_s(x_{r_2}, A)+r_2=-\frac{k(n)}{r_2^s}-2\int_{B_{r_1}} \frac{1}{|x_{r_2}-y|^{n+s}}dy+r_2.  
\end{cases}
\end{equation}
We define the function $g(r_1, r_2)= \left(g_1(r_1, r_2), g_2(r_1, r_2)\right)$ as follows:
\[\begin{cases} g_1(r_1, r_2):= \frac{k(n)}{r_1^s}+2\lim_{\eps\to 0}\left(\int_{B_{r_1-\eps}} \frac{1}{|x_{r_1}-y|^{n+s}}dy-
\int_{B_{r_2}\setminus B_{r_1+\eps}} \frac{1}{|x_{r_1}-y|^{n+s}}dy\right)+r_1\\  
g_2(r_1, r_2):=-\frac{k(n)}{r_2^s}-2\int_{B_{r_1}} \frac{1}{|x_{r_2}-y|^{n+s}}dy+r_2.\end{cases}\]

We now   compute the Jacobian matrix $Dg$ at the point $(\bar r_1, \bar r_2)$ which is  a stationary point for \eqref{sistemaanello}, that is $g(\bar r_1, \bar r_2)=0$. 

We observe the following fact: for  $\delta\neq 0$, $\eps>0$, $R>r>|\delta|$,  
  there hold 
\begin{eqnarray*} \int_{B_{r+\delta-\eps}} \frac{1}{|x_{r+\delta}-y|^{n+s}}dy&=&
\left(\frac{r}{r+\delta}\right)^s\int_{B_{r-\eps\frac{r}{r+\delta}}} \frac{1}{|x_r-y|^{n+s}}dy\\
\int_{B_R\setminus B_{r+\delta+\eps}} \frac{1}{|x_{r+\delta}-y|^{n+s}}dy
&=&\left(\frac{r}{r+\delta}\right)^s\int_{B_{\frac{Rr}{r+\delta}}\setminus B_{r+\eps\frac{r}{r+\delta}}} \frac{1}{|x_r-y|^{n+s}}dy.
\end{eqnarray*} 
So, using these equalities we  get that  the derivative of $g_1$ at $(\bar r_1, \bar r_2)$ are given by  
\begin{eqnarray} \nonumber \partial_{r_1} g_1(\bar r_1, \bar r_2)&=&  -\frac{sk(n)}{\bar r_1^{s+1}}-\frac{2s}{\bar r_1}\lim_{\eps\to 0}\left( \int_{B_{\bar r_1-\eps}} \frac{1}{|x_{\bar r_1}-y|^{n+s}}dy-\int_{B_{\bar r_2}\setminus B_{\bar r_1+\eps}} \frac{1}{|x_{\bar r_1}-y|^{n+s}}dy\right)\\ &+&\frac{2\bar r_2}{\bar r_1}  \int_{\partial B_{\bar r_2}}  \frac{1}{|x_{\bar r_1}-y|^{n+s}}dy+1 \\\nonumber &=& -\frac{s}{r}g_1(\bar r_1, \bar r_2)+s+1+\frac{2\bar r_2}{\bar r_1}  \int_{\partial B_{\bar r_2}}  \frac{1}{|x_{\bar r_1}-y|^{n+s}}dy\\
&=&  s+1+\frac{2\bar r_2}{\bar r_1}  \int_{\partial B_{\bar r_2}}  \frac{1}{|x_{\bar r_1}-y|^{n+s}}dy\
\label{derivata1} 
 \\\partial_{r_2} g_1(\bar r_1, \bar r_2)&=& -2 \int_{\partial B_{\bar r_2}}  \frac{1}{|x_{\bar r_1}-y|^{n+s}}dy.\nonumber  \end{eqnarray}

Analogously, we observe that 
for  $\delta\neq 0$,  $R>r>|\delta|$,  
there holds
\[\int_{B_r} \frac{1}{|x_{R+\delta}-y|^{n+s}}dy
=\left(\frac{R}{R+\delta}\right)^s\int_{B_{\frac{Rr}{R+\delta}}} \frac{1}{|x_R-y|^{n+s}}dy.
\]
Using this equality, we  compute the derivative of $g_2$ at $(\bar r_1, \bar r_2)$:
\begin{eqnarray} \nonumber \partial_{r_1} g_2(\bar r_1, \bar r_2)&=& -2\int_{\partial B_{\bar r_1}}  \frac{1}{|x_{\bar r_2}-y|^{n+s}}dy\\
\partial_{r_2} g_2(\bar r_1, \bar r_2)&=& \frac{sk(n)}{\bar r_2^{s+1}}+\frac{2s}{\bar r_2}\int_{B_{\bar r_1}} \frac{1}{|x_{\bar r_2}-y|^{n+s}}dy+\frac{2\bar r_1}{\bar r_2}\int_{\partial B_{\bar r_1}}  \frac{1}{|x_{\bar r_2}-y|^{n+s}}dy+1\nonumber \\
&=& 
-\frac{s}{\bar r_2}g_2(\bar r_1,\bar r_2)+s+1+\frac{2\bar r_1}{\bar r_2}\int_{\partial B_{\bar r_2}}  \frac{1}{|x_{\bar r_2}-y|^{n+s}}dy\nonumber\\
&=& s+1+\frac{2\bar r_1}{\bar r_2}\int_{\partial B_{\bar r_2}}  \frac{1}{|x_{\bar r_2}-y|^{n+s}}dy
.\label{derivata2}
\end{eqnarray}

Note that, using \eqref{derivata1} and \eqref{derivata2}, 
\[Dg(\bar r_1,\bar r_2)(\bar r_1,\bar r_2)^t= (s+1)(\bar r_1,\bar r_2)^t\] 
so that $(\bar r_1,\bar r_2)$ is an eigenvector with eigenvalue $s+1>0$.
Moreover, by \eqref{derivata2}, we observe that $\partial_{r_2} g_2(\bar r_1, \bar r_2)>s+1$. 
This implies that
\begin{equation}\label{auto}
\max_{v:\,|v|=1} v Dg(\bar r_1, r_2) v^t\geq (0,1)Dg(\bar r_1, \bar r_2)(0,1)^t > s+1,
\end{equation}
which gives that 
$Dg(\bar r_1, r_2) $  has a second eigenvalue bigger than $s+1$, and then in particular positive. 

%
%
 
 \smallskip
 
 We now consider the general case of a self-shrinker 
 \begin{equation}\label{selfgen}
E:= \bigcup_{k=1}^N \left(B_{\bar r_{2k}}\setminus B_{\bar r_{2k-1}}\right).
\end{equation}
We also let $\bar r=(\bar r_1,\ldots,\bar r_{2N}),\,g(\bar r)=(g_1(\bar r),\ldots,g_{2N}(\bar r)) \in \R^2N$, where
\begin{eqnarray*}
g_{i}(\bar r):= -H_s(x_{i},E)+\bar r_{i}&=& -\frac{k(n)}{\bar r_{i}^s} + \bar r_{i}
+2\sum_{j< i}  (-1)^{i-j} \int_{B_{\bar r_j}}\frac{1}{|x_{\bar r_i}-y|^{n+s}}\,dy
\\
&& -2\sum_{j> i}  (-1)^{i-j} \int_{\R^n\setminus B_{\bar r_j}}\frac{1}{|x_{\bar r_i}-y|^{n+s}}\,dy,
\end{eqnarray*}
 if the index $i$ is even, and 
 \begin{eqnarray*}
g_i(\bar r)&:=& H_s(x_i,E)+r_i =
-H_s(x_i,\R^n\setminus E)+r_i 
\\
&=& - \frac{k(n)}{\bar r_{i}^s} + \bar r_{i}
+2\sum_{j< i}  (-1)^{i-j} \int_{B_{\bar r_j}}\frac{1}{|x_{\bar r_i}-y|^{n+s}}\,dy
\\
&& -2\sum_{j> i}  (-1)^{i-j} \int_{\R^n\setminus B_{\bar r_j}}\frac{1}{|x_{\bar r_i}-y|^{n+s}}\,dy,
\end{eqnarray*}
if $i$ is odd. Notice that, since $\bar r$ is a stationary solutions to \eqref{syst}, we have $g(\bar r)=0$. 

We compute, for $j\ne i$, 
\[
\frac{\partial g_i}{\partial r_j}(\bar r) = 2 (-1)^{i-j}  \int_{\partial B_{\bar r_j}}\frac{1}{|x_{\bar r_i}-y|^{n+s}}\,dy,
\]
and
 \begin{eqnarray*}
\frac{\partial g_i}{\partial r_i}(\bar r) &=& - \frac{s k(n)}{\bar r_{i}^{s+1}} + 1 
-2\frac{s}{\bar r_i}\sum_{j< i}  (-1)^{i-j} \int_{B_{\bar r_j}}\frac{1}{|x_{\bar r_i}-y|^{n+s}}\,dy
\\
&& +2\frac{s}{\bar r_i}\sum_{j> i}  (-1)^{i-j} \int_{\R^n\setminus B_{\bar r_j}}\frac{1}{|x_{\bar r_i}-y|^{n+s}}\,dy
 -\frac{2}{\bar r_i}\sum_{j\ne i}  (-1)^{i-j} \bar r_j \int_{\partial B_{\bar r_j}}\frac{1}{|x_{\bar r_i}-y|^{n+s}}\,dy
\\
&=& -\frac{2}{\bar r_i} g_i(\bar r) + s+ 1  -\frac{2}{\bar r_i}\sum_{j\ne i}  (-1)^{i-j} \bar r_j \int_{\partial B_{\bar r_j}}\frac{1}{|x_{\bar r_i}-y|^{n+s}}\,dy
\\
&=&  s+ 1  +2 \sum_{j\ne i}  \frac{\bar r_j}{\bar r_i}\, (-1)^{i-j+1} \int_{\partial B_{\bar r_j}}\frac{1}{|x_{\bar r_i}-y|^{n+s}}\,dy.
\end{eqnarray*}
Notice that
\[
Dg(\bar r) \bar r^t = \sum_j  \frac{\partial g_i}{\partial r_j}(\bar r) r_j = (s+1) \bar r^t,
\]
so that $\bar r$ is an eigenvector with eigenvalue $s+1>0$.

Now we claim that \begin{equation}
\label{clamipo} \frac{\partial g_{2N}}{\partial r_{2N}}(\bar r) > s+1.
\end{equation} 
If the claim is true, then reasoning as in \eqref{auto}, we conclude that there exists 
 an eigenvalue of $Dg(\bar r)$ which  is strictly greater than $s+1$ (and then positive), so that the Morse index of $(\bar r_1,\ldots \bar r_{2N})$ is at least $2$. 

Since 
\[
\frac{\partial g_{2N}}{\partial r_{2N}}(\bar r) = s+1 +\frac{2}{\bar r_{2N}}
\sum_{j=1}^{2N-1} (-1)^{j-1} \,\bar r_{2N-j} \int_{\partial B_{\bar r_{2N-j}}}\frac{1}{|x_{\bar r_{2N}}-y|^{n+s}}\,dy,
\]
to get the claim \eqref{clamipo} 
it is sufficient to prove that  
for  all $1\leq i<j<2N$ there holds
 \begin{equation}\label{mono}
\bar r_i \int_{\partial B_{\bar r_i}}\frac{1}{|x_{\bar r_{2N}}-y|^{n+s}}\,dy <
\bar r_j \int_{\partial B_{\bar r_j}}\frac{1}{|x_{\bar r_{2N}}-y|^{n+s}}\,dy.
\end{equation}
We shall prove a slightly stronger statement, namely that 
\begin{equation}\label{minni}
r\mapsto h(r) := \int_{\partial B_{r}}\frac{1}{|x_{\bar r_{2N}}-y|^{n+s}}\,dy\quad \text{ is strictly increasing on $(0,\bar r_{2N})$. }
\end{equation}
Indeed, we compute
\begin{eqnarray*}
h'(r) &=&  \int_{\partial B_{r}} \nabla \left( \frac{1}{|x_{\bar r_{2N}}-y|^{n+s}}\right) \cdot \nu(y)\,dy= \int_{B_{r}} \Delta \left( \frac{1}{|x_{\bar r_{2N}}-x|^{n+s}}\right) \,dx
\\ &=&  (n+s) (s+2) \int_{B_{r}}  \frac{1}{|x_{\bar r_{2N}}-x|^{n+s+2}}\,dx >0,
\end{eqnarray*}
which shows \eqref{minni}, and so proves \eqref{mono}.

\end{proof} 
 
 \begin{remark}
\upshape It would be interesting to determine exactly the Morse index  of the stationary points $(\bar r_1,\bar r_2, \dots, \bar r_{2N})$ (resp. $(\bar r_0,\bar r_1, \dots, \bar r_{2N})$)  of the flow \eqref{syst}. 
  In the simplest case $N=1$, we proved in Theorem \ref{teostab} that the index of $(\bar r_1,\bar r_2)$ is 
  equal to $2$. 

 It would also be interesting  to understand if the ball is dynamically stable for any perturbation, not necessarily radial,
 as it happens for the standard mean curvature flow \cites{MR772132,cm}.
 \end{remark}
 
 
  \end{document}